\documentclass[a4paper,12pt]{amsart}

\usepackage[paper=a4paper,textwidth=400pt,textheight=620pt,includeheadfoot,twoside=True]{geometry}

\usepackage[dvipdfm,svgnames]{xcolor}
\usepackage{graphics}

\usepackage[active]{srcltx}
\usepackage{amsmath,amssymb, amsthm,amsxtra}

\usepackage[mathscr]{eucal}

\usepackage{bm}
\usepackage[all]{xy}

\usepackage{aliascnt}

\theoremstyle{plain}
\newtheorem{thm}{Theorem}[section]
\newtheorem*{thm*}{Theorem}

\newaliascnt{prop}{thm}
\newaliascnt{cor}{thm}
\newaliascnt{lem}{thm}
\newaliascnt{claim}{thm}
\newaliascnt{defn}{thm}
\newaliascnt{ques}{thm}
\newaliascnt{conj}{thm}
\newaliascnt{fact}{thm}
\newaliascnt{rem}{thm}
\newaliascnt{ex}{thm}
\newtheorem{prop}[prop]{Proposition}

\newtheorem{lem}[lem]{Lemma}

\newtheorem*{prop*}{Proposition}
\newtheorem*{cor*}{Corollary}
\newtheorem*{lem*}{Lemma}
\newtheorem*{claim*}{Claim}
\theoremstyle{definition}

\newtheorem*{defn*}{Definition}
\newtheorem*{ques*}{Question}
\newtheorem*{conj*}{Conjecture}

\newtheorem*{prob*}{Problem}

\newtheorem{rem}[rem]{Remark}
\newtheorem{ex}[ex]{Example}
\newtheorem*{fact*}{Fact}
\newtheorem*{rem*}{Remark}
\newtheorem*{ex*}{Example}
\aliascntresetthe{prop}
\aliascntresetthe{cor}
\aliascntresetthe{lem}
\aliascntresetthe{claim}
\aliascntresetthe{defn}
\aliascntresetthe{ques}
\aliascntresetthe{conj}
\aliascntresetthe{fact}
\aliascntresetthe{rem}
\aliascntresetthe{ex}
\usepackage[pointedenum]{paralist}
\usepackage{varioref}
\labelformat{equation}{\textnormal{(#1)}}
\labelformat{enumi}{\textnormal{(#1)}}

\setdefaultenum{(a)}{(i)}{(1)}{(A)}%
\usepackage{hyperref}

\def\textsectionN~{\textsection{}}

\makeatletter
\renewcommand{\O}{\@undefined}
\renewcommand{\P}{\@undefined}
\makeatother

\renewcommand\phi{\varphi}
\renewcommand\epsilon{\varepsilon}
\renewcommand\leq{\leqslant}

\makeatletter
\newcommand{\set}{%
  \@ifstar{\@setstar}{\@set}%
}%
\newcommand{\@setstar}[2]{\{\, #1 \mid #2 \,\}}
\newcommand{\@set}[1]{\{#1\}}
\makeatother
\newcommand{\lin}[1]{\langle\, #1 \,\rangle}
\newcommand{\trans}[1][1]{\raisebox{#1ex}{\scriptsize\kern0.1em$t$\kern-0.1em}}

\newcommand{\PP}{\mathbb{P}}

\newcommand{\TT}{\mathbb{T}}
\newcommand{\PN}{\PP^N}
\newcommand{\Pv}[1][N]{(\PP^{#1})\spcheck}

\newcommand{\sO}{\mathscr{O}}

\DeclareMathOperator{\rk}{rk}%
\DeclareMathOperator{\Hom}{Hom}%
\DeclareMathOperator{\im}{im}%
\newcommand{\Gr}{\mathbb{G}}
\newcommand\spcirc{^\circ}
\newcommand{\tdiff}[2]{{\partial #1}/{\partial #2}}
\newcommand{\diff}{\tdiff}

\newcommand{\sQ}{\mathscr{Q}}
\newcommand{\sS}{\mathscr{S}}
\newcommand{\sHom}{\mathscr{H}om}

\newcommand\Xo{X\spcirc}
\newcommand\Xdo{{X'}\spcirc}
\newcommand\xo{x_o}
\newcommand{\yo}{y_o}
\newcommand{\Gn}{\Gr(n,\PN)}

\newcommand\LL{\mathbb{L}}
\newcommand\LLx[1][x]{\LL_{#1}(\ker d_{#1}\gamma)}

\title[On general fibers of Gauss maps]{On general fibers of Gauss maps \\ in positive characteristic}%

\hypersetup{%
 pdftitle={On general fibers of Gauss maps in positive characteristic},%
 pdfauthor={Katsuhisa FURUKAWA}%
}
\email{katu@tims.ntu.edu.tw}
\author[K.~Furukawa]{Katsuhisa~FURUKAWA}
\address{
  Department of Mathematics, Center for Advanced Studies in %
  Theoretical Sciences, National Taiwan University, %
  Taipei, 
  Taiwan}
\subjclass[2010]{Primary 14N05; Secondary 14M15}
\keywords{Gauss map, degeneracy map, strange variety}
\date{}

\begin{document}

\maketitle

\begin{abstract}
  A general fiber of the Gauss map of a projective variety in $\PN$
  coincides with a linear subvariety of $\PN$ in characteristic zero.
  In positive characteristic, S. Fukasawa showed that
  a general fiber of the Gauss map can be a non-linear variety.
  In this paper, we show that each irreducible component of such a possibly non-linear fiber
  of the Gauss map
  is contracted to one point by the degeneracy map,
  and is contained in a linear subvariety corresponding
  to the kernel of the differential of the Gauss map.
  We also show the inseparability of Gauss maps of strange varieties not being cones.
\end{abstract}

\vspace{1em}

\section{Introduction}
\label{sec:introduction}

Let $X \subset \PN$ be an $n$-dimensional projective variety over an algebraically closed field of arbitrary characteristic,
and let $\gamma: X \dashrightarrow \Gn$ be the \emph{Gauss map} of $X$,
which sends a smooth point $x \in X$ to the embedded tangent space $\TT_xX$ to $X$ at $x$ in $\PN$.
We denote by
$d_x \gamma: T_xX \rightarrow T_{\gamma(x)} \Gn$
the differential of $\gamma$ at $x \in X$,
a linear map between Zariski tangent spaces at $x$ and $\gamma(x)$.
We denote by $\rk\gamma$
the rank of $d_x \gamma$ at general $x \in X$.
Note that $\gamma$ is separable if and only if $\rk\gamma = \dim(\im(\gamma))$.

The \emph{degeneracy map} $\kappa$ of $X$ is defined to be
the rational map
\begin{equation*}  \kappa: X \dashrightarrow \Gr({n - \rk\gamma}, \PN)
\end{equation*}
which sends a general point $x \in X$ to the $({n - \rk\gamma})$-plane $\LLx \subset \PN$,
where we denote by $\LL_x(A) \subset \TT_xX$
the $m$-plane corresponding to an $m$-dimensional vector subspace $A \subset T_xX$.

\vspace{1em}
In characteristic zero,
it is well known that the closure of a general fiber of $\gamma$
is equal to a linear subvariety of $\PN$
(P.~Griffiths and J.~Harris \cite[(2.10)]{GH}, F.~L.~Zak \cite[I, 2.3.~Theorem~(c)]{Zak};
S.~L.~Kleiman and R.~Piene gave another proof in terms of reflexivity \cite[pp.~108--109]{KP}),
where the fiber of $\gamma$
is indeed equal to $\kappa(x) = \LLx \subset \PN$ for general $x$ in the fiber.
The same statement holds
in positive characteristic
if $\gamma$ is separable
\cite[Theorem~1.1]{expshr}.

A.~H.~Wallace \cite[\textsection 7]{Wallace} pointed out that
the Gauss map $\gamma$ can be \emph{inseparable} in positive characteristic.
In this case,
a general fiber of $\gamma$ 
is not equal to $\kappa(x)$,
since their dimensions are different.
Moreover, it is possible that
a general fiber of $\gamma$ is \emph{not} equal to a linear subvariety of $\PN$;
the fiber can be a union of points
(Wallace \cite[\textsection 7]{Wallace}, Kleiman and A.~Thorup \cite[I-3]{Kleiman86},
H.~Kaji \cite[Example~4.1]{Kaji1986} \cite{Kaji1989}, J.~Rathmann \cite[Example~2.13]{Rathmann}, A.~Noma \cite{Noma2001}),
and can be a union of non-linear varieties
(S.~Fukasawa \cite[\textsection{}7]{Fukasawa2005} \cite{Fukasawa2006}, the author and A.~Ito \cite[\textsection{}5]{FI} \cite[Theorem~1.3]{FI2}).

In this paper, we investigate the relationship between
the $({n - \rk\gamma})$-plane $\kappa(x) = \LLx$
and the general fiber of $\gamma$ (possibly non-linear, as above).

\begin{thm}\label{thm:mainthm-1}
  Let $X \subset \PN$ be an $n$-dimensional projective variety,
  and let $F \subset X$ be an irreducible component of the closure of
  a general fiber of the Gauss map $\gamma$.
  Then $\LLx \subset \PN$
  is constant on general $x \in F$
  (in other words, $F$ is contracted to one point by $\kappa$).
  Therefore $F$ is contained in this constant $({n - \rk\gamma})$-plane.
\end{thm}

The above constant $({n - \rk\gamma})$-plane in $\PN$
corresponds to $\kappa(F)$.
We note that $\kappa(F_1) \neq \kappa(F_2)$ can occur
for two irreducible components $F_1$ and $F_2$ of a general fiber of $\gamma$
(see \autoref{thm:ex-fib-diff-pl}).

\vspace{1em}

Next we examine
Gauss maps of strange varieties.
A projective variety $X \subset \PN$ is said to be \emph{strange} for a point $v \in \PN$ if $v \in \TT_xX$ holds for any smooth point $x \in X$.
In previous studies of Gauss maps in positive characteristic,
the inseparability of $\gamma$ of strange $X$
was often observed with attractive phenomena (e.g., \cite{Fukasawa2006}, \cite{fukasawa-kaji3}).
Motivated by such observations, we show:

\begin{thm}\label{thm:mainthm-2}\label{thm:mainthm-strange}
  Let $X \subset \PN$ be a projective variety
  which is strange for a point $v \in \PN$.
  Then $v$ is contained in $\LLx$.
  Moreover, if $X$ is {not} a cone with vertex $v$,
  then $\gamma$ is inseparable.
\end{thm}

The paper is organized as follows.
In \autoref{sec:restr-gauss-maps-1},
we first show that the image $d_x\gamma(T_xX)$ in $T_{\gamma(x)}\Gn$ is constant 
for $x$ in an irreducible component $F$ of a general fiber of $\gamma$.
In \autoref{sec:shrinking-map-with},
by using the techniques on the shrinking map with respect to the Gauss map,
which is studied in \cite[\textsection{}5]{FI2},
we investigate a sheaf description for $\kappa$.
Then we have that $\kappa(x)$ is constant on $x \in F$, and
prove \autoref{thm:mainthm-1}.
In \autoref{sec:gauss-maps-strange-1}, we investigate the case where $X$ is strange for $v$.
This means that $\gamma(X) \subset \Gr_v$,
where $\Gr_v \simeq \Gr(n-1, \PP^{N-1})$ is the set of $n$-planes containing $v$.
Using the sheaf description for $\kappa$ again, we prove \autoref{thm:mainthm-strange}.
This theorem
gives some typical examples in \autoref{sec:examples}, which are relevant to
\autoref{thm:mainthm-1}.

\section{Irreducible component of a general fiber of the Gauss map}

We will show the following two propositions and then prove \autoref{thm:mainthm-1}
in the end of this section.
Let $X \subset \PN$ be a projective variety, and let $\gamma: X \dashrightarrow \Gn$ be the Gauss map of $X$.

\begin{prop}\label{thm:const-dgammaTxX-xinF}
  Let $\xo \in X$ be a general point, and let $\yo := \gamma(\xo)$.
  Let $F$ be an irreducible component of $\overline{\gamma^{-1}(\yo)}$
  such that $\xo \in F$.
  Then we have $d_{\xo}\gamma(T_{\xo}X) = d_x\gamma(T_xX)$
  in $T_{\yo}\Gn$ for general $x \in F$.
\end{prop}

\begin{prop}\label{thm:kappa-const-xinF}
  Let $F$ be as in \autoref{thm:const-dgammaTxX-xinF}.
  Then $\kappa(x)$ is constant on general $x \in F$.
\end{prop}

\subsection{Restriction of the Gauss map}
\label{sec:restr-gauss-maps-1}

Let $r:= \rk \gamma$, and let
$\alpha \in X$ be a general point such that $\rk d_\alpha \gamma = r$.
Take an $(N-n+r)$-plane $A \subset \PN$ containing $\alpha$,
such that $T_{\alpha} A \cap \ker d_{\alpha} \gamma = 0$ in $T_{\alpha}\PN$,
and that $T_{\alpha} A \cap T_{\alpha}X$ is of dimension $r$.
Then $X \cap A$ is smooth at $\alpha$.

Let $X''$ be the irreducible component of $X \cap A$ containing $\alpha$.
Then $X''$ is smooth at $\alpha$ and $\dim (X'') = r$.

\begin{lem}\label{thm:gammaX''A}
  Let $Y' := \gamma(X'') \subset \Gn$.
  Then $\gamma|_{X''}: X'' \dashrightarrow Y'$ is separable and generically finite; in particular, $\dim Y' = r$.
\end{lem}
\begin{proof}
  It follows from $T_{\alpha} X'' \cap \ker(d_\alpha\gamma) = 0$.
\end{proof}

\begin{lem}\label{thm:gammaX'A}
  Let $X' \subset X$ be an irreducible component of
  $\overline{\gamma^{-1}Y'}$ containing $X''$.
  Then $\gamma|_{X'}: X' \dashrightarrow Y'$ is separable.
\end{lem}
\begin{proof}
  Let $x$ be a general point of $X''$.
  Since $\gamma|_{X''}$ is separable,
  the composite homomorphism
  $\Omega_{Y', \gamma(x)}^1 \rightarrow \Omega_{X', x}^1 \rightarrow \Omega_{X'', x}^1$
  is injective;
  hence so is the homomorphism
  $\Omega_{Y', \gamma(x)}^1 \rightarrow \Omega_{X', x}^1$.
  Thus $\gamma^*\Omega_{Y'}^1 \rightarrow \Omega_{X'}^1$ is injective on an open neighborhood of $x$ in $X'$,
  i.e., $\gamma|_{X'}$ is separable.
\end{proof}

\begin{rem} Since $\alpha \in X$ is general, we can assume that
  $\alpha$ is a smooth point of $X$
  and that $\gamma(\alpha)$ is a smooth point of $Y$.
  Then
  $X' \cap X^{sm} \cap \gamma^{-1}(Y^{sm}) \neq \emptyset$,
  where we denote by $X^{sm}$ the smooth locus of $X$, and so on.
\end{rem}

Now, let us consider general fibers of $\gamma$.

\begin{lem}\label{thm:F-in-X'}\label{thm:F-in-X'-rem}
  Let $\Xdo \subset X'$ to be the non-empty open subset
  \[
  \set*{x \in {X'}^{sm}}{\rk d_x (\gamma|_{X'}) = r} \cap X^{sm} \cap \gamma^{-1}(Y^{sm}) \cap \gamma^{-1}({Y'}^{sm})\setminus \bigcup_{i=1}^{s} V_i,
  \]
  where
  $V_1, \dots, V_s$ are the irreducible components of
  $\overline{\gamma^{-1}Y'}$ not equal to $X'$.
  Let $x \in \Xdo$
  and let $F$ be an irreducible component of $\overline{\gamma^{-1}(\gamma(x))} \subset X$ containing $x$.
  Then $F$ must be contained in $X'$.
\end{lem}
\begin{proof}
  Since $F$ is irreducible and contained in $\overline{\gamma^{-1}Y'}$,
  $F$ is contained in $X'$ or $V_i$.
  Since $x$ is not contained in $\bigcup V_i$,
  and since $x \in F$,
  we have $F \subset X'$.
\end{proof}

\begin{rem}\label{thm:unionX'o}
  For general $\alpha \in X$, we can take $X' = X'_{\alpha}$ as above.
  Thus the union of ${X'_{\alpha}}\spcirc$'s becomes a dense subset of $X$.
\end{rem}

\begin{proof}[Proof of \autoref{thm:const-dgammaTxX-xinF}]
  Let $\xo \in X$ be a general point and let $\yo := \gamma(\xo)$.
  We may take some $\alpha \in X$ such that $\xo \in {X'_{\alpha}}\spcirc$
  as in \autoref{thm:unionX'o}, and then
  the irreducible component $F$ of $\overline{\gamma^{-1}(\yo)}$ is contained in $X'$
  due to \autoref{thm:F-in-X'}.
  Then 
  $d_{\xo}\gamma(T_{\xo}X) = d_{\xo}\gamma(T_{\xo}X') = T_{\yo}Y'$.
  Moreover, since $\gamma(F) = \yo$,
  it holds that $d_{x}\gamma(T_{x}X) = d_{x}\gamma(T_{x}X') = T_{\yo}Y'$ for general $x \in F$.
\end{proof}

\subsection{Shrinking map with respect to the Gauss map}
\label{sec:shrinking-map-with}

In the proof of \autoref{thm:kappa-const-xinF},
the shrinking map $\sigma_{\gamma}$ with respect to the Gauss map $\gamma$
will play a key role.
We prepare some notations and recall the definition of $\sigma_{\gamma}$
(for the general case, see \cite[Definition 2.1]{FI2}).

Let $\sQ$ and $\sS$ be
the universal quotient bundle and subbundle of rank $n+1$ and $N-n$ on $\Gn$
with the exact sequence
$0 \rightarrow \sS \rightarrow H^0(\PN, \sO(1))
\otimes \sO_{\Gn}\rightarrow \sQ \rightarrow 0$.
For an open subset $\Xo \subset X^{sm}$,
we have the following composite homomorphism
\[
\Phi : \gamma^*\sQ\spcheck \rightarrow \gamma^*\sHom(\sHom(\sQ\spcheck , \sS\spcheck), \sS\spcheck)
\rightarrow \sHom(T_{\Xo}, \gamma^*\sS\spcheck),
\]
where
the first homomorphism
is induced from the dual of
$\sS \otimes \sS\spcheck \rightarrow \sO$,
and the second one
is induced from the differential
\[
d\gamma : T_{\Xo} \rightarrow \gamma^*T_{\Gn} = \gamma^*\sHom(\sQ\spcheck , \sS\spcheck).
\]
Set $n^{-} := \rk(\ker\Phi) -1$, where $\rk(\ker\Phi)$ is the rank as a torsion free sheaf.
Then $\ker \Phi|_{\Xo}$ is a subbundle of
$H^0(\PN, \sO(1))\spcheck \otimes \sO_{\Xo}$
of rank $n^{-}+1$
(replacing $\Xo \subset X$ by an smaller open subset if necessary).
By the universality of the Grassmann variety, 
we have an induced morphism
\[
\sigma_{\gamma}: \Xo \rightarrow  \Gr(n^{-}, \PN)
\]
and call it the \emph{shrinking map with respect to $\gamma$}.

We note that
the $n^{-}$-plane $\sigma_{\gamma}(x)$ in $\PN$ corresponds to
the projectivization of $\ker \Phi \otimes k(x) \subset H^0(\PN, \sO(1))\spcheck$
for $x \in \Xo$.

\begin{rem}\label{thm:kappa-sigma}
  From \cite[Proposition 5.2]{FI2},
  $n^{-} = n-\rk\gamma$ and
  $\kappa = \sigma_{\gamma}$.
\end{rem}

\begin{proof}[Proof of \autoref{thm:kappa-const-xinF}]
  Let $\xo \in X$, $\yo := \gamma(\xo)$, and $F \subset \overline{\gamma^{-1}(\yo)}$
  be as in \autoref{thm:const-dgammaTxX-xinF}. Then,
  setting $L := d_{\xo}\gamma (T_{\xo}X)$ in $T_{\yo}\Gn$,
  we have $L = d_{x}\gamma (T_{x}X)$ for general $x \in F$.
  Thus the linear map
  $\Phi_x$ is expressed with the following commutative diagram:
  \[
  \xymatrix{    Q \spcheck \ar@/^1pc/@<1ex>[rr]^{\phi} \ar@/_1pc/[drr]_{\Phi_x} \ar[r] & \Hom(T_{\yo}\Gn, S\spcheck) \ar[r]
    &\Hom(L, S\spcheck) \ar[d]^{- \circ d_x\gamma} \\
    &&\Hom(T_{x}X, S\spcheck),
  }  \]
  where we write $Q := \sQ \otimes k(\yo)$ and $S := \sS \otimes k(\yo)$.
  Since the vertical arrow of the above diagram is injective,
  we have $\ker(\Phi_x) = \ker(\phi)$,
  where $\phi: Q \spcheck \rightarrow \Hom(L, S\spcheck)$ in the diagram is independent to the choice of $x \in F$.
  This means that $\sigma_{\gamma}(x)$,
  which corresponds to the projectivization of $\ker(\Phi_x)$,
  is constant on general $x \in F$.
  Since $\kappa(x) = \sigma_{\gamma}(x)$ as mentioned in \autoref{thm:kappa-sigma},
  the assertion follows.
\end{proof}

Now let us conclude the proof of \autoref{thm:mainthm-1}.

\begin{rem}\label{thm:Z-to-Y}
  Let $Z \subset X$ be a closed subset with $Z \neq X$, i.e., $\dim(Z) < n$.
  Then each irreducible component $F$ of a general fiber of $\gamma: X \dashrightarrow Y$
  is not contained in $Z$.
  Otherwise,
  the dimension of a general fiber of $Z \dashrightarrow Y$
  is equal to $n-\dim(Y)$, which implies $\dim Z = n$, a contradiction.
\end{rem}

\begin{proof}[Proof of \autoref{thm:mainthm-1}]
  Let $\Xo$ be a dense subset consisting of $x \in X$ satisfying the
  conditions of \autoref{thm:const-dgammaTxX-xinF}.
  Then each irreducible component $F$ of a general fiber of $\gamma$
  satisfies $F \cap \Xo \neq \emptyset$,
  as in \autoref{thm:Z-to-Y} with $Z = \overline{X \setminus \Xo}$.
  We can apply Propositions \ref{thm:const-dgammaTxX-xinF} and \ref{thm:kappa-const-xinF}
  by taking some $\xo \in F \cap \Xo$, and then $\kappa$ contracts $F$ to one point.
\end{proof}

\section{Gauss maps of strange varieties}
\label{sec:gauss-maps-strange}

In this section, we will prove \autoref{thm:mainthm-2} and use it
to construct typical examples relevant to \autoref{thm:mainthm-1}.

\subsection{Strange point and the kernel of the differential of the Gauss map}
\label{sec:gauss-maps-strange-1}

We continue to use the notations in \autoref{sec:shrinking-map-with}.
Let $\pi_v: \PN \setminus \set{v} \rightarrow \PP^{N-1}$ be the linear projection from $v$,
yielding an inclusion
\begin{equation*}  \Gr_v := \Gr(n-1, \PP^{N-1}) \hookrightarrow \Gn
\end{equation*}
which sends an $(n-1)$-plane $L \subset \PP^{N-1}$ to
the $n$-plane $\pi_v^{-1}(L) \cup \set{v} \subset \PN$.
We note that $X$ is strange for $v$
if and only if $\gamma(X) \subset \Gr_v$.

We denote by
$\sQ$ and $\sS$
the universal quotient bundle and subbundle of ranks $n+1$ and $N-n$ on $\Gn$,
and by
$\sQ_v$ and $\sS_v$
the universal quotient bundle and subbundle of ranks $n$ and $N-n$ on $\Gr_v$.

\begin{rem}\label{thm:v-E}
  \begin{inparaenum}
  \item 
    We have
    $\sS|_{\Gr_v} = \sS_v$ and have a natural injection
    $\sQ_v \hookrightarrow \sQ|_{\Gr_v}$
    with the following commutative diagram:
    \[
    \xymatrix{      0 \ar[r]& \sS_v \ar@{=}[d] \ar[r]& H^0(\PP^{N-1}, \sO(1)) \otimes \sO_{\Gr_v} \ar@{^(->}[d] \ar[r]
      & \sQ_v \ar@{^(->}[d]\ar[r]& 0
      \\
      0 \ar[r]& \sS|_{\Gr_v} \ar[r]& H^0(\PN, \sO(1)) \otimes \sO_{\Gr_v} \ar[r]& \sQ|_{\Gr_v} \ar[r]& 0
      \makebox[0pt]{\ .}
    }    \]

  \item \label{item:thm:v-E-2}
    Let $E$ be the kernel of the surjection $\sQ|_{\Gr_v}\spcheck \twoheadrightarrow {\sQ_v}\spcheck$ 
    induced from the dual of the above injection. Then $E$ is of rank $1$.
    Let $y = [M] \in \Gn$ be an $n$-plane belonging to $\Gr_v$.
    Then the projectivization of $E \otimes k(y) \subset \sQ|_{\Gr_v}\spcheck \otimes k(y)$
    in $H^0(\PN, \sO(1))$
    corresponds to $v \in M$ in $\PN$.
  \end{inparaenum}
\end{rem}

\begin{proof}[Proof of \autoref{thm:mainthm-2}]
  Since $\gamma(X) \subset \Gr_v$,
  the homomorphism $\Phi$ given in \autoref{sec:shrinking-map-with}
  is described with the following commutative diagram:
  \[
  \xymatrix{    \gamma^*\sQ|_{\Gr_v}\spcheck \ar[r] \ar@/^1pc/@<1ex>[rr]^{\Phi} \ar@{->>}[d]
    & \gamma^*\sHom(\sHom(\sQ|_{\Gr_v}\spcheck , \sS_v\spcheck), \sS_v\spcheck)
    \ar@{->>}[d] \ar[r]
    & \sHom(T_{\Xo}, \gamma^*\sS_v\spcheck) \ar@{=}[d]
    \\
    \gamma^*\sQ_v\spcheck \ar[r]& \gamma^*\sHom(\sHom(\sQ_v\spcheck , \sS_v\spcheck), \sS_v\spcheck)
    \ar[r]& \sHom(T_{\Xo}, \gamma^*\sS_v\spcheck),
  }  \]
  where the middle vertical arrow is induced from
  $T_{\Gr_v} = \sHom(\sQ_v\spcheck , \sS_v\spcheck) \hookrightarrow T_{\Gn}|_{\Gr_v}$.
  From the diagram, we have $\gamma^*E \subset \ker\Phi$, where
  $E := \ker(\sQ|_{\Gr_v}\spcheck \twoheadrightarrow {\sQ_v}\spcheck)$.
  Therefore the point $v \in \PN$, which corresponds to $E$ as in \autoref{thm:v-E}\ref{item:thm:v-E-2},
  is contained in $\sigma_{\gamma}(x) = \kappa(x) = \LLx$.
  Hence the former statement of the theorem follows.

  If $\gamma$ is separable, then $X$ must be a cone with vertex $v$.
  The reason is as follows.
  Let $x \in X$ be general. From \cite[Theorem 1.1]{expshr} (or \cite[Corollary 5.4]{FI2}),
  $F:= \overline{\gamma^{-1}(\gamma(x))} \subset X$ is a linear variety.
  Then $F = \LL_x(\ker d_x\gamma)$
  (this is because, $T_xF \subset \ker d_x\gamma$ and their dimensions coincide).
  Therefore the line $\overline{xv}$ is contained in $F \subset X$, which means that $v$ is a vertex.
\end{proof}

\begin{rem}\label{thm:rem-strange-loc}
  We say that $X$ is \emph{strange} for a linear variety $V \subset \PN$ if
  $V \subset \TT_xX$ for any smooth point $x \in X$,
  equivalently, $X$ is strange for any point $v \in V$.
  If $X$ is strange for $V$, then $\LLx$ contains $V$ for general point $x \in X$.
  This immediately follows from \autoref{thm:mainthm-strange},
  since $v \in \LLx$ for any $v \in V$.
\end{rem}

\subsection{Examples}
\label{sec:examples}

As mentioned in \autoref{sec:introduction},
according to Fukasawa's results, a general fiber of an inseparable Gauss map
can be a non-linear variety. We check \autoref{thm:mainthm-1} in such a case:

\begin{ex}\label{thm:fib-conic-LLx}
  Let $K$ be the ground field with $\mathrm{char}K = 3$,
  and let $X \subset \PP^4$ be the hypersurface defined by 
  $f = Z_1^6+Z_2^6+Z_3Z_4Z_0^4,$
  where $(Z_0: Z_1: \dots: Z_4)$ are the homogeneous coordinates on $\PP^4$.
  Then the following holds.
  \begin{enumerate}
  \item The Gauss map $\gamma$ of $X$ is inseparable,
    and its general fiber is set-theoretically equal to a smooth conic $C \subset X$.
  \item $\LLx$ is constant for general $x$ in such a conic $C$,
    and moreover, is equal to the $2$-plane spanned by $C$.

  \end{enumerate}
  The reason is as follows.
  We can identify $\gamma: X \dashrightarrow \Pv[4]$ with
  the rational map sending $(Z_0: \dots: Z_4)$ to
  \[
  (\diff{f}{Z_i})_{0\leq i\leq4}
  = (Z_3Z_4Z_0^3: 0: 0: Z_4Z_0^4: Z_3Z_0^4).
  \]
  Let $\ell = (Z_0 = Z_3 = Z_4 = 0)$, a line in $\PP^4$.
  Then $X$ is strange for $\ell$.
  The image of $\gamma$ is equal to $\ell^* \subset \Pv[4]$,
  the set of hyperplanes containing $\ell$. Here
  $\ell^*$ is a $2$-plane of $\Pv[4]$.
  On the other hand, the rank of $\gamma$ is equal to $1$.
  For a general point $x \in X$,
  it follows from \autoref{thm:mainthm-strange} and \autoref{thm:rem-strange-loc} that
  \[
  \LLx = \lin{x, \ell},
  \]
  where the right hand side is the $2$-plane in $\PP^4$ spanned by $x$ and $\ell$.

  Now, we fix a general point $\alpha = (1: \alpha_1: \dots: \alpha_4) \in X$ such that
  $\alpha_3, \alpha_4$ are nonzero. Then we have
  $\gamma(\alpha) = (1: 0: 0: 1/\alpha_3: 1/\alpha_4)$ in $\Pv[4]$.
  On the other hand, the $2$-plane
  $\LLx[\alpha] = \lin{\alpha, \ell}$ is defined by
  $\alpha_3Z_0 - Z_3, \alpha_4Z_0 - Z_4$. 
  Let us consider $C_{\alpha} = \lin{\alpha, \ell} \cap X$, whose defining polynomials are
  $Z_1^6+Z_2^6+\alpha_3\alpha_4Z_0^6, \alpha_3Z_0 - Z_3, \alpha_4Z_0 - Z_4$.
  Then $C_{\alpha} \subset \PP^4$ is set-theoretically equal to the smooth conic in $\lin{\alpha, \ell}$,
  \[
  (Z_1^2+Z_2^2+\sqrt[3]{\alpha_3\alpha_4}Z_0^2= \alpha_3Z_0 - Z_3= \alpha_4Z_0 - Z_4 = 0),
  \]
  where $\mathrm{char}K = 3$ and $\sqrt[3]{\alpha_3\alpha_4} \in K$ is the unique third root of ${\alpha_3\alpha_4}$.

  We have $\gamma(C_{\alpha}) = \set{\gamma(\alpha)}$
  since the coordinates of a general point of $C_{\alpha}$ is written by $(1: *: *: \alpha_3: \alpha_4)$.
  If a point $\beta = (1: \beta_1: \dots: \beta_4) \in X$ satisfies
  $\gamma(\alpha) = \gamma(\beta)$, then we have $\beta_3 = \alpha_3, \beta_4 = \alpha_4$,
  and then $C_{\alpha} = C_{\beta}$ holds because of their defining polynomials.
  Hence $\overline{\gamma^{-1}(\gamma(\alpha))}$ coincides with
  the conic $C_{\alpha}$, and the statement of (a) holds.
  For general $x \in C_{\alpha}$, the $2$-plane
  $\lin{x, \ell}$ is spanned by $C_{\alpha}$, which implies the statement of (b).
\end{ex}

In \autoref{thm:mainthm-1},
in the case where a general fiber of the Gauss map is not irreducible,
for irreducible components $F$'s of the fiber,
the constant $({n - \rk\gamma})$-planes $\kappa(F)$'s 
can be different:

\begin{ex}\label{thm:ex-fib-diff-pl}
  Let $\mathrm{char}K = 3$, and let
  $X \subset \PP^3$ be the surface defined by the homogeneous polynomial,
  $f = Z_0^5 + Z_1^5 - Z_2^3 Z_3^2$.
  Then the following holds.
  \begin{enumerate}
  \item
    The Gauss map $\gamma$ of $X$ is inseparable, and its general fiber
    is set-theoretically equal to a set of $4$ points.

  \item\label{eq:ex-LneqL}
    $\LL_x(\ker d_x\gamma) \neq \LL_{x'}(\ker d_{x'}\gamma)$
    for distinct $x,x'$ of the above $4$ points.
  \end{enumerate}
  The reason is as follows.
  The map $\gamma: X \dashrightarrow \Pv[3]$ is given by
  \begin{equation}\label{eq:ex-diff}
    (\diff{f}{Z_i})_{0 \leq i \leq 3} = 
    (- Z_0^4 : - Z_1^4 : 0 : Z_2^3Z_3).
  \end{equation}
  Thus $X$ is strange for $v := (0:0:1:0)$,
  and the rank of $\gamma$ is equal to $1$.
  In addition, $\gamma$ is generically finite; thus,
  each irreducible component of a general fiber of $\gamma$
  is a set of one point.

  Let $x = (1:a:b:c) \in X$ be a general point such that $a\neq 0$.
  Let $x' = (w: a': b': c') \in X$ be a point satisfying $\gamma(x') = \gamma(x)$.
  By \ref{eq:ex-diff}, we have
  \[
  (-w^4 : -{a'}^4: 0 : {b'}^3c') = (-1 : -a^4 : 0 : b^3c).
  \]
  Since $w$ must be nonzero, we can set $w=1$ and express $x' = (1: a': b': c')$.
  Then the above equality implies
  ${a'}^4 = a^4$ and ${b'}^3c' = b^3c$.
  Thus $a'$ is a $4$-th root of $a^4$,
  which is equal to a $4$-th root of $1$ multiplied by $a$.
  We recall that the set of $4$-th roots of $1$ consists of $4$ elements
  $1, -1, \xi, -\xi$,  where $\xi \in K$ be the square root of $-1$.

  Let $\zeta \in K$ be one of the $4$-th roots of $1$, and set $a' = a \zeta$.
  We show that $b', c' \in K$ are uniquely determined, as follows.
  Since $f(x) = f(x') = 0$, we have $1+a^5-b^3c^2 = 1+(a\zeta)^5-{b'}^3{c'}^2 = 0$.
  Thus,
  \[
  a^5\zeta = \zeta(b^3c^2-1) = {b'}^3{c'}^2-1.
  \]
  It follows from $(b')^3c' = b^3c$ that
  $\zeta(b^3c^2-1) ={b'}^3{c'}^2-1 = b^3cc' - 1$, which implies
  \[
  c' = \frac{\zeta(b^3c^2-1) + 1}{b^3c}.
  \]
  Hence we also have 
  \[
  {b'}^3 = -\frac{b^6c^2}{\zeta(b^3c^2-1) + 1}.
  \]
  Thus $b'$ is obtained as the unique third root of
  the right hand side of the above formula.

  We set $x_{\zeta} := (1: a \zeta: b': c') \in X$ for a $4$-th root $\zeta \in \set{1, -1, \xi, -\xi}$ of $1$,
  where $b',c' \in K$ are 
  uniquely determined by $\zeta$
  as above.
  Then,
  \[
  \gamma^{-1}(\gamma(x)) = \set{x, x_{-1}, x_{\xi}, x_{-\xi}}.
  \]
  For the strange point $v := (0:0:1:0)$,
  we have $\overline{xv} \neq \overline{x'v}$ for each $x' \in \gamma^{-1}(\gamma(x))$ with $x' \neq x$.
  On the other hand, \autoref{thm:mainthm-strange} implies that
  $\LL_x(\ker d_x\gamma) = \overline{xv}$ and $\LL_{x'}(\ker d_{x'}\gamma) = \overline{x'v}$.
  Since the two lines are distinct, the assertion of \ref{eq:ex-LneqL} follows.
\end{ex}

\subsection*{Acknowledgments}
The author was partially supported by JSPS KAKENHI Grant Number 25800030.

\vspace{1ex}

\begin{thebibliography}{9}

\bibitem{Fukasawa2005}
  S.~Fukasawa, Developable varieties in positive characteristic. Hiroshima Math. J. {\bf 35} (2005), 167--182.

\bibitem{Fukasawa2006}
  \bysame, Varieties with non-linear Gauss fibers. Math. Ann. {\bf 334} (2006), 235--239.


\bibitem{fukasawa-kaji3} S. Fukasawa and H. Kaji, Any algebraic variety in positive characteristic admits a projective model with an inseparable Gauss map, J. Pure Appl. Algebra {\bf 214} (2010), 297--300. 


  
  
  
  
  
\bibitem{expshr}
  K. Furukawa, Duality with expanding maps and shrinking maps, and its applications to Gauss maps, Math.\ Ann. \textbf{358} (2014), 403--432.

\bibitem{FI}
  K. Furukawa and A. Ito, Gauss maps of toric varieties, arXiv:1403.0793.

\bibitem{FI2}
  \bysame, On Gauss maps in positive characteristic in view of images, fibers, and field extensions, arXiv:1412.7023.
  
\bibitem{GH}
  P. Griffiths and J. Harris, Algebraic geometry and local differential geometry. Ann. Sci. Ecole Norm.
  Sup. (4) {\bf 12} (1979), 355--432.

  
  
  
  
  
  
  
\bibitem{Kaji1986} H.~Kaji, On the tangentially degenerate curves. J. London Math. Soc. (2) {\bf 33} (1986), 430--440.

\bibitem{Kaji1989} \bysame, On the Gauss maps of space curves in characteristic $p$. Compositio Math. {\bf 70} (1989), 177--197.

  
  
  
  
  
  
  
  

  
  
\bibitem{Kleiman86} S. L. Kleiman, Tangency and duality, ``Proceedings of the 1984 Vancouver conference in algebraic geometry'', CMS Conference Proceedings 6, Amer. Math. Soc., Providence, 1986, pp.163--226.

\bibitem{KP} S. L. Kleiman and R. Piene, On the inseparability of the Gauss map.
  ``Enumerative Algebraic Geometry
  (Proceedings of the 1989 Zeuthen Symposium),'' Contemp. Math. {\bf 123}, pp. 107--129.
  Amer. Math. Soc., Providence, 1991.

  
  
\bibitem{Noma2001} A.~Noma, Gauss maps with nontrivial separable degree in positive characteristic. J. Pure Appl. Algebra {\bf 156} (2001), 81--93.

  
  
\bibitem{Rathmann} J.~Rathmann, The uniform position principle for curves in characteristic $p$. Math. Ann. {\bf 276} (1987), 565--579.

  
\bibitem{Wallace} A. H.~Wallace, Tangency and duality over arbitrary fields. Proc. London Math. Soc. (3) {\bf 6} (1956), 321--342.

\bibitem{Zak}
  F.~L.~Zak, \emph{Tangents and secants of algebraic varieties}.
  Transl. Math. Monographs {\bf 127},  Amer. Math. Soc., Providence, 1993. 

\end{thebibliography}
\end{document}